\documentclass[10pt,a4paper,oneside,no-math]{amsart}
\usepackage[a4paper,margin=2.5cm]{geometry}
\usepackage[dvipsnames]{xcolor}
\usepackage[hypertexnames=false,colorlinks=true,linkcolor=blue,%
citecolor=purple,filecolor=magenta,urlcolor=cyan]{hyperref}                      
\usepackage{amssymb,anyfontsize,enumerate,layout,mathrsfs,mdwlist,stmaryrd,thmtools,tikz}
\usepackage{dsfont} 
\usepackage[textsize=footnotesize]{todonotes}
\presetkeys%
    {todonotes}%
    {color=Apricot}{}%
\usepackage[style=alphabetic,maxnames=99,maxalphanames=5]{biblatex}
\usepackage{mathtools}
\usepackage[english,noabbrev]{cleveref}

\allowdisplaybreaks 

\theoremstyle{plain}                          
\newtheorem{theorem}{Theorem}[section]
\newtheorem{proposition}[theorem]{Proposition}    
\newtheorem{lemma}[theorem]{Lemma}
\newtheorem{corollary}[theorem]{Corollary}
\theoremstyle{definition}
\newtheorem{definition}[theorem]{Definition}
\newtheorem{prop-defin}[theorem]{Proposition-definition} 
\newtheorem{example}[theorem]{Example}
\theoremstyle{remark}
\newtheorem{remark}[theorem]{Remark}

\newcommand{\normord}[1]{\vcentcolon\mathrel{#1}\vcentcolon}
\newcommand{\mb}[1]{\mathbb{#1}} 
\newcommand{\mf}[1]{\mathfrak{#1}}
\newcommand{\mc}[1]{\mathcal{#1}}

\newcommand{\1}{\mathds{1}}

\newcommand{\C}{\mb{C}} 
\newcommand{\Z}{\mb{Z}} 
\newcommand{\Q}{\mb{Q}}
\renewcommand{\P}{\mb{P}}
\newcommand{\M}{\mc{M}}
\newcommand{\Mgn}{\overline{\mc{M}}_{g,n}}

\newcommand{\del}{\partial}

\DeclareMathOperator{\End}{End}
\DeclareMathOperator{\Aut}{Aut}

\newcommand{\Id}{\mathord{\mathrm{Id}}}
\DeclareMathOperator*{\Res}{Res}

\DeclareRobustCommand{\Stirling}{\genfrac\{\}{0pt}{}}

\begingroup\expandafter\expandafter\expandafter\endgroup
\expandafter\ifx\csname pdfsuppresswarningpagegroup\endcsname\relax
\else
\fi
{}

\begin{document}

\title{KP hierarchy for Hurwitz-type cohomological field theories}

\author[R.~Kramer]{Reinier Kramer}
\address{University of Alberta, Edmonton, AB, T6G 2G1, Canada}
\email{reinier@ualberta.ca}

\thanks{}

\begin{abstract}
	We generalise a result of Kazarian regarding Kadomtsev-Petviashvili integrability for single Hodge integrals to general cohomological field theories related to Hurwitz-type counting problems or hypergeometric tau-functions. The proof uses recent results on the relations between hypergeometric tau-functions and topological recursion, as well as the DOSS correspondence between topological recursion and cohomological field theories. As a particular case, we recover the result of Alexandrov of KP integrability for triple Hodge integrals with a Calabi-Yau condition.
\end{abstract}

\maketitle

\tableofcontents


\section{Introduction}

The moduli spaces of curves are a central object in modern algebraic geometry, and have been studied intensively. In particular, their intersection theory is a subject of ongoing research. The space $ \Mgn $ has $ n$ line bundles $ \mb{L}_i $ whose fibres at a point are the cotangent lines at the $ i$th point of the represented curve, and a rank-$ g$ Hodge bundle $ \mb{E} $ whose fibres are the space of one-forms at the curve. Their Chern classes are defined to be $ \psi_i \coloneqq c_1(\mb{L}_i) $ and $ \lambda_j \coloneqq c_j (\mb{E}) $, respectively. Moreover, the spaces $\Mgn$ for different $g$ and $n$ have many structure maps between them, and many classes behave well under these maps. A collection of classes on all $ \Mgn$ satisfying certain coherence axioms with respect to the structure maps are called cohomological field theories (CohFTs), and these play an important role in enumerative geometry of curves. One well-known example is the total Hodge class $ \Lambda (t) = \sum \lambda_i t^i $.\par
By the Witten-Kontsevich theorem~\cite{Witten,Kontsevich}, moduli spaces of curves have many relations to areas of mathematical physics and integrable hierarchies. In particular, this theorem proves that a generating function of the intersection numbers of $ \psi $-classes is a tau-function of the Korteweg-de Vries hierarchy.\par
Furthermore, the Ekedahl-Lando-Shapiro-Vainshtein formula~\cite{ELSV} relates single Hodge integrals, i.e. intersection numbers of $\Lambda (-1)$ with $ \psi$-classes, to simple single Hurwitz numbers, counting ramified coverings of $ \P^1_\C $ with only simple ramifications (with profile $ (2,1,1,1,\dotsc )$) except for one point. Hurwitz numbers themselves also give a large class of tau-functions of Toda or Kadomtsev-Petviashvili hierarchies (of which the KdV hierarchy is a reduction), as noted by Okounkov~\cite{Okounkov00}.\par
Kazarian~\cite{KazarianHodge} interpreted the ELSV formula as a change of variables from the generating function of single Hodge integrals to a tau-function of the Kadomtsev-Petviashvili hierarchy, using the result of Okounkov on simple single Hurwitz numbers.\par
All of these results have strong relations to Chekhov-Eynard-Orantin topological recursion \cite{CEO06,EO07}, a succesful way of encoding many counting problem with a natural genus expansion into a spectral curve with a recursively defined collection of multidifferentials, which should be generating functions of the counts. The Witten-Kontsevich $\psi$-intersection numbers can be encoded this way, and this is somehow the base case of the theory. Many types of Hurwitz numbers obey topological recursion as well, starting with \cite{BM08,BEMS11} for the first case of simple Hurwitz numbers, and culminating in the works of Bychkov-Dunin-Barkowski-Kazarian-Shadrin~\cite{BDKS20,BDKS20a}, which prove topological recursion for two large families of hypergeometric KP tau-functions, encompassing nearly all previously-studied cases of Hurwitz numbers.\par
In another direction, there is a general correspondence between topological recursion and intersection numbers of CohFTs \cite{Eyn14,DOSS14}, which vastly generalises the ELSV formula when combined with the results on topological recursion for Hurwitz numbers.\par
A particularly interesting case is the conjecture of Mari\~{n}o and Vafa~\cite{MarinoVafa} on a further generalisation of the ELSV formula, proved independently in \cite{LLZ03,OP04}. This Mari\~{n}o-Vafa formula relates triple Hodge integrals with a Calabi-Yau condition to topological vertex amplitudes, i.e. Gromov-Witten invariants of $ \C^3 $. Topological recursion was conjectured for toric Calabi-Yau threefolds by Bouchard-Klemm-Mariño-Pasquetti~\cite{BKMP09}. It was first proved in \cite{Che18,Z09} for $ \C^3$, as well as in \cite{Eyn11} as an example of the general correspondence of \cref{EynardDOSS}, while the general BKMP conjecture was proved in \cite{EyOr}.\par
Both the space of CohFTs and the space of KP tau-functions have an action of an infinite-dimensional group, respectively the Givental group and the Heisenberg-Virasoro group. As certain elements of these spaces have been identified by Witten-Kontsevich and Kazarian, and different integrable hierarchies have been constructed for general CohFTs by Dubrovin-Zhang~\cite{DubrovinZhang-NormalForms} and Buryak~\cite{Bur15}, one may ask how general the relation is with KP specifically, and the group actions are a natural tool to study this question.\par
Alexandrov~\cite{Ale20} showed that in the case of a rank-one CohFT, the orbits of the Witten-Kontsevich CohFT/tau-function under these two different group actions have an intersection which is only two-dimensional, and contains exactly the triple Hodge integrals that appear in the Mariño-Vafa formula. As a consequence, Alexandrov generalises Kazarian's result to show that the generating function of Calabi-Yau triple Hodge integrals satisfies the KP hierarchy after a linear change of variables.

\subsection*{Results of this paper}
We give a new viewpoint on the relation found by Alexandrov, by generalising Kazarian's proof in \cite{KazarianHodge} to all hypergeometric KP tau-functions satisfying topological recursion, using the above results. This yields a change of variables coming from the function $X$ for any hypergeometric tau-function preserving the KP hierarchy after removing the unstable terms of the tau-function. When topological recursion holds, this resulting tau-function can be interpreted as the generating function of the cohomological field theory.\par
In general, the change of variables contains infinite linear combinations. However, we identify when the linear combinations are actually finite, and find a finite-dimensional family for each CohFT rank. In the rank one case, this recovers exactly the triple Hodge integrals, in a particular parametrisation. For higher rank, this family seems to fit within Alexandrov's deformed generalised Kontsevich model~\cite{Ale21}.\par
Interestingly, the function $X$ may also be a Möbius transformation. In this case, there is no unstable correction term, and this can be interpreted as certain independence of the parametrisation of the spectral curve. This also resolves the meaning behind Kazarian's change of coordinates, as voiced in \cite[Remark 2.6]{KazarianHodge}: ``The definition for the change (6) looks unmotivated. [...] The only motivation that we can provide here is that `it works'.'' There is quite a freedom of choice, but the particular choice Kazarian made reduces to the finite-dimensional family indicated above.

\subsection*{Open questions}

Single and triple Hodge integrals have been studied intensively in relation to Dubrovin-Zhang hierarchies, yielding relations to the intermediate long wave (ILW) hierarchy and the fractional Volterra hierarchy, cf. \cite{Buryak-ILW1,Buryak-ILW2,LYZZ21}. The relation between those results and the current work are still unclear, and will be discussed elsewhere. Between the  first and second versions of this preprint, Liu-Wang-Zhang \cite{LWZ21} related the  ILW hierarchy to a limit of fractional Volterra hierarchy viewed as a reduction of the 2D Toda hierarchy, possibly giving a new avenue to relating to the current paper.\par
The family where the linear change of variables is finite seems like an interesting and natural deformation of Witten's $r$-spin class, keeping a single ramification point, but splitting the pole of $dx$. However, this family seems mostly unknown, with the exception of Alexandrov's work mentioned above. It may be interesting to investigate it more closely, in order to better understand the deformation of higher-order zeroes of $dx$.\par
Currently, there is a gap in the literature on limits of spectral curves, which in particular limits the validity of the proof \cref{OScorrelators}, and hence the applicability of the main theorem of this paper, to $dx$ with simple zeroes. Future work with Borot, Bouchard, Chidambaram, and Shadrin will fix this, and will investigate more generally the applicability of limit arguments for topological recursion.\par
For the BKP hierarchy, similar results should hold. In particular, Alexandrov and Shadrin~\cite{AS21} proved an adapted topological recursion for a large class of hypergeometric BKP tau-functions, analogous to \cref{OScorrelators}. The analogous ELSV-Eynard-DOSS correspondence between this kind of topological recursion and cohomological field theories has not appeared in the literature, but the special case of completed cycles spin Hurwitz numbers is treated in work of the author with Giacchetto and Lewa\'{n}ski~\cite{GKL21}.

\subsection*{Outline of the paper}

\Cref{KazKP} contains prerequisites. In \cref{KPintro,InfGrass}, we give a short introduction to the Kadomtsev-Petviashvili hierarchy and its space of solutions. In \cref{RecallHodge}, we recall the main ideas from~\cite{KazarianHodge}, which we will generalise. In \cref{HypGeoIntro,TRCohFT}, we recall recent results on hypergeometric tau-functions and their relations to topological recursion and cohomological field theories, and state our main theorem, which is \cref{KPforHurwitzCohFT}. We also introduce, in \cref{MVandKP}, the generating function of triple Hodge numbers, which is the main motivating example of this paper.\par
In \cref{sec:KPforTH}, we prove the main result. Firstly, in \cref{changeofvars}, we find a change of variables, for any hypergeometric tau-function, that preserves the property of being a tau-function after removal of unstable terms, \cref{KPtauWithoutUnstable}. In \cref{KPwithTR}, we restrict to the case where topological recursion holds, and use this machinery to obtain tau-functions of intersection numbers, proving our main result. We also determine, in \cref{Finiteness}, the exact conditions for the change of variables to be finite, in a specific sense. Finally, in \cref{THsection}, we return to the triple Hodge integrals, and prove an explicit version of the main theorem for this case.

\subsection*{Notation} 

We work over the field of complex numbers $\C$. We will use the function $ \varsigma (z) \coloneqq e^{\frac{z}{2}} - e^{-\frac{z}{2}} $ and $ \mc{S}(z) = \frac{\varsigma (z)}{z}$, $ \mu $ and $ \nu $ will denote partitions, and $ z_\mu \coloneqq \prod_{i=1}^{\mu_1} i^{m_i(\mu )}m_i (\mu )! $, where $ m_i (\mu )$ is the number of parts of $ \mu $ of size $ i$. We will also consistently write $ n \coloneqq \ell (\mu ) $ and $ \llbracket n \rrbracket \coloneqq \{ 1, \dotsc, n \}$.

\subsection*{On the origin of this paper}

An earlier version of this text, only concerning triple Hodge integrals, was written in 2018, shortly after A. Alexandrov informed me of his result. That version appeared in my PhD dissertation \cite[Chapter 10]{Kra19}. This paper is an updated and extended version of that chapter.

\subsection*{Acknowledgments}

I would like to thank A. Alexandrov for informing me of his theorem on triple Hodge integrals, S. Shadrin for introducing me to the subject and suggesting generalising Kazarian's method to this case, and both of them, G. Carlet, N. Chidambaram, and A. Giacchetto for many interesting discussions. I would also like to thank A. Alexandov and S. Shadrin for suggesting improvements to a previous version of this paper, and P. Norbury for finding errors by calculating \cref{NaiveHodge}, spurring me to double-check and do this example myself.\par
Research conducted for this paper is supported by the Netherlands Organization for Scientific Research, the Max-Planck-Gesellschaft, the Natural Sciences and Engineering Research Council of Canada, and the Pacific Institute for the Mathematical Sciences (PIMS). The research and findings may not reflect those of these institutions.\par
The University of Alberta respectfully acknowledges that we are situated on Treaty 6 territory, traditional lands of First Nations and Métis people.

\section{Prerequisites on the KP hierarchy and topological recursion}\label{KazKP}

In this section, we review some standard notions on the KP hierarchy and its relations to the infinite Grassmannian. We give the main outline of Kazarian's proof of KP for single Hodge integrals, which we will use as a blueprint for our results. We also recall the class of hypergeometric tau-functions, which fulfills a central role in this paper, as well as its relation to topological recursion and cohomological field theories. Finally, we recall the Mariño-Vafa formula for triple Hodge integrals and show it fits in the setup.

\subsection{The KP hierarchy}\label{KPintro}

The Kadomtsev-Petviashvili hierarchy is an infinite set of evolutionary differential equations in infinitely many variables. It is a very well-studied system, and some introductions into the subject can be found in \cite{DickeyBook,Kharchev,MJD}.\par
Let $ \underline{t} = \{ t_i \}_{i\geq 1} $ be a set of independent variables and $ \del \coloneqq \frac{\del}{\del t_1} $. Define the pseudo-differential operator (i.e. a Laurent series in $ \del^{-1} $ with coefficients functions in $ \underline{t} $ with composition defined formally)
\begin{equation}\label{LaxOp}
	L = \del + u_1 \del^{-1} + u_2 \del^{-2} + \dotsc\,.
\end{equation}
where the $ u_j $ are dependent variables in the $ t_i $. For a pseudo-differential operator $ O$, define $ O_+ $ to be its purely differential part, the part without powers of $ \del^{-1} $. The \emph{Lax formulation of the KP hierarchy} is given by the system of equations
\begin{equation}\label{LaxForm}
	\frac{\del L}{\del t_i} = \big[ (L^i)_+,L\big]\,.
\end{equation}
This is a system of partial differential equations for the $ u_j $, and they can be interpreted as the compatibility equations for the system
\begin{align}
	L\Psi &= z \Psi & \frac{\del \Psi}{\del t_i} = (L^i)_+\Psi\,.
\end{align}
The function $ \Psi $ is called the \emph{Baker-Akhiezer} function. The first non-trivial equation, the KP equation, is
\begin{equation}
	3\frac{\del^2 u_1}{\del t_2^2} - \frac{\del}{\del t_1} \Big( 4 \frac{\del u_1}{\del t_3} -12 u_1 \frac{\del u_1}{\del t_1} - \frac{\del^3 u_1}{\del t_1^3} \Big) = 0 \,.
\end{equation}
The Baker-Akhiezer function can be written in the form
\begin{equation}
	\Psi = \frac{\tau \big( \big\{ t_k - \frac{z^{-k}}{k} \big\} \big)}{\tau (\{ t_k \} )} e^{\xi (\underline{t}, z)} \,, \qquad \xi (\underline{t},z) =  \sum_{k=1}^\infty t_k z^k \,.
\end{equation}
Here $ \tau $ is a single function, called a tau-function, dependent on the $ t_k $, and all dependent variables can be expressed in terms of this one function. This way, the entire hierarchy can be rewritten as bilinear equations for $ \tau $ called \emph{Hirota equations}:
\begin{equation}
	\Res_{z = \infty} dz \, e^{\xi (\underline{t}, z) - \xi (\underline{t}', z)} \tau \big( \big\{ t_k - \frac{1}{k z^k} \big\} \big) \tau \big( \big\{ t_k' + \frac{1}{k z^k} \big\} \big)\,.
\end{equation} 
Writing $ F = \log \tau$, we find $ u_1 = \del^2 \log \tau$, and the first two equations are
\begin{align}
	0 &= 3 \frac{\del^2 F}{\del t_2^2} - 4  \frac{\del^2 F}{\del t_3 \del t_1}  + \frac{\del^4 F}{\del t_1^4} + 6 \Big( \frac{\del^2 F}{\del t_1^2} \Big)^2\,;\\
	0 &= 2 \frac{\del^2 F}{\del t_3 \del t_2} - 3 \frac{\del^2 F}{\del t_4 \del t_1} + \frac{\del^4 F}{\del t_2 \del t_1^3} + 6 \frac{\del^2 F}{\del t_2 \del t_1} \frac{\del^2 F}{\del t_1^2}\,.
\end{align}




\subsection{Space of tau-functions and Lie action}\label{InfGrass}

The space of solutions of the KP hierarchy is an infinite-dimensional Grassmannian \cite{Sato82}, which is usually Pl\"ucker embedded in a Fock space, i.e. a highest weight module of a certain Clifford algebra. The Hirota equations are then the Plücker relations defining the Grassmannian inside the Fock space. By the boson-fermion correspondence, this can also be expressed in terms of symmetric functions, which is the viewpoint we will adopt here.

\begin{definition}
	We write $ \Lambda \coloneqq \C \llbracket p_1, p_2, \dotsc \rrbracket$ for the \emph{space of symmetric functions}, also called the \emph{bosonic Fock space} (of type A). Here the $p_k$ are power-sum functions $ p_k = \sum_i X_i^k$ in some countably infinite variable set $ \underline{X} =  \{ X_i\}$.\par
	For other symmetric functions in $ \underline{X} $, e.g. the Schur functions $ s_\lambda$, we write $ s_\lambda (\underline{p}) \coloneqq s_\lambda (\underline{X})$.
\end{definition}

The space of symmetric functions has a projective action of the Lie algebra $ \mf{gl}(\infty) $, the algebra of infinite square matrices $ (a_{ij})_{i,j \in \Z + \frac{1}{2}} $.\footnote{In order to make the Lie bracket well-defined, some decay condition is needed. A common choice is restriction to finitely many diagonals, but there are other options, see e.g. \cite{Sato82}. We will remain agnostic on this choice, as in this paper, the required convergence in guaranteed by our constructions.} This space has a standard basis given by $ E_{kl} = (\delta_{ik}\delta_{jkl})_{ij} $. 
Define the \emph{vertex operator}
\begin{equation}
	Z(z,w) = \frac{1}{z-w} \bigg( \exp \Big( \sum_{j=1}^\infty (z^j-w^j)p_j\Big) \exp \Big( -\sum_{k=1}^\infty (z^{-k}-w^{-k})\frac{1}{k} \frac{\del}{\del p_k} \Big) -1 \bigg)\,.
\end{equation}
Then expanding this vertex operator as
\begin{equation}
 Z(z,w) = \sum_{i,j \in \Z + \frac{1}{2}} Z_{ij} z^{i+1/2} w^{-j-1/2}\,,
\end{equation}
the assignment $ E_{ij} \mapsto Z_{ij} $ is a projective representation of $ \mf{gl}(\infty )$, i.e. a representation of a central extension $ \widehat{\mf{gl}(\infty)}$.\par
The matrices $ \alpha_k = \sum_{l \in \Z + \frac{1}{2}} E_{l-k,l} $ give rise to the following operators on $ \Lambda $:
\begin{equation}
	a_k \coloneqq  \begin{cases}p_k & k > 0\\ -k \frac{\del}{\del p_{-k}} & k <0\\ 0 & k =0 \end{cases}\,.
\end{equation}
We also define the following operators:
\begin{align}
	L_m &\coloneqq \frac{1}{2}\sum_{i=-\infty}^\infty \normord{a_i a_{m-i}}\,, & M_l &\coloneqq \frac{1}{6} \sum_{i,j = -\infty}^\infty \normord{a_i a_j a_{l-i-j}}\,,
\end{align}
where the $ \normord{ \;} $, the \emph{normal ordering}, means one should order the operators inside in order of decreasing index. All of these operators are in $ \widehat{\mf{gl}(\infty)} $.\par

\begin{theorem}[\cite{Sato82}]
	Under the identification $ t_k = \frac{p_k}{k}$, the space of KP tau-functions is the orbit of $ 1 \in \Lambda$ under the action of $ \widehat{\mf{gl}(\infty)}$.
\end{theorem}


\subsection{Single Hodge integrals}\label{RecallHodge}

In~\cite{KazarianHodge}, Kazarian considered the generating function for single Hodge integrals,
\begin{equation}
	F_\textup{H} (u; T_0,T_1,T_2,\dotsc ) \coloneqq \sum_{g,n} \frac{1}{n!} \sum_{d_1, \dotsc, d_n \geq 0} \int_{\Mgn} \Lambda (-u^2)\prod_{i=1}^n \psi_i^{d_i} T_{d_i}\,,
\end{equation}
and showed that its exponent, $ Z_\textup{H} \coloneqq \exp (F_\textup{H}) $, is a tau-function for the KP hierarchy, after a certain change of coordinates. Explicitly, this change of coordinates is given as follows: define
\begin{equation}\label{KazarianRecursion}
	D = (u+ z)^2 z \frac{\del}{\del z}\,.
\end{equation}
Then we define the $ T_d $ in terms of other coordinates $ q_k $ by the linear correspondence
\begin{align}\label{KazarianTDef}
	q_k &\leftrightarrow z^k\,,  & T_d \leftrightarrow  D^d  z\,.
\end{align}
The proof consists of three steps, and makes essential use of the ELSV formula \cite{ELSV} to transform this generating function into a generating function of Hurwitz numbers.\par
The first step, \cite[Theorem 2.2]{KazarianHodge}, is the observation that the generating function for single simple Hurwitz numbers is a tau-tunction for the KP hierarchy. This is a well-known result, see \cite{Okounkov00}. In fact, the single simple Hurwitz generating function can be obtained from the trivial $ \tau $-function $ 1$ by the action of two very explicit elements of the Lie group associated to $ \widehat{\mf{gl}(\infty )} $. 
The second step, \cite[Theorem 2.3]{KazarianHodge}, uses the ELSV formula to rewrite the Hurwitz generating function (after subtracting the unstable geometries) as a generating function for single Hodge integrals. This introduces certain combinatorial factors, that suggest a certain change of coordinates, which is encoded by the equation $ X(z) = \frac{z}{1+\beta z} e^{-\frac{\beta z}{1+\beta z}}$.  After this change of coordinates, we obtain $ Z_\textup{H} $, viewed as a function in $ q$'s.\par
The third step, \cite[Theorem 2.5]{KazarianHodge} shows that a certain class of coordinate changes preserves solutions of the KP hierarchy, after they are modified with a quadratic function. In essence, this coordinate change is given infinitesimally by the flow along the differential part of an $ A \in \widehat{\mf{gl}(\infty )}$, whose polynomial part is exactly the added quadratic function. In this specific case, this quadratic function is exactly the $ (0,2)$ part of the Hurwitz generating function.\par
\vspace{12pt}
In this paper, we will generalise this proof scheme to a more general setting. We will start from a general hypergeometric tau-function in the sense of \cref{OScorrelators} below, corresponding to the first point of the proof. \par
We obtain a change of coordinates coming from this formalism that can always be completed to an automorphism of KP when correcting with the $H_{0,2}$ of \cref{OSUnstable}, without any further assumption, corresponding to the third point of the proof.\par
If we restrict to the class of hypergeometric tau-functions satisfying topological recursion, we can use the correspondence between topological recursion and cohomological field theories of  Eynard and Dunin-Barkowski-Orantin-Shadrin-Spitz~\cite{Eyn14,DOSS14}, which generalises the ELSV formula and hence gives the second step.\par
In the particular case of triple Hodge integrals, the role of the ELSV formula is taken by the Mari\~no-Vafa formula.  For explanations on all the required notions and notation, see the following sections.

\subsection{Hypergeometric KP tau-functions and topological recursion}\label{HypGeoIntro}

An important class of KP tau-functions is given by the hypergeometric tau-functions \cite{KMMM95,OS01a,OS01b}, for which we will use the results and notation of \cite{BDKS20}. In two large families of examples, these satisfy Eynard-Orantin topological recursion~\cite{EO07} (or its generalisation to non-simple ramification given by Bouchard-Eynard~\cite{BE13}), which we define first. We will confine ourselves to the case of rational spectral curves, as this is the appropriate setting for the Hurwitz-type problems covered.

\begin{definition}[\cite{EO07,BE13}]
	A \emph{rational spectral curve} is a quadruple $\mc{C} = (\Sigma = \P^1, dx,dy,B)$, where $ dx$ and $ dy$ are meromorphic one-forms on $ \Sigma$ with no common zeroes, only simple poles of $dx$, and without poles of $ dy$ at zeroes of $dx$, and $ B = B(z_1,z_2) =\frac{dz_1 \, dz_2}{(z_1-z_2)^2}$ is a symmetric $(1,1)$-form on $ \Sigma \times \Sigma$. Write $ R \subset \Sigma$ for the set of zeroes of $ dx$, and $r_a$ for the order of vanishing of $ dx$ at $ a \in R$.\par
	On a rational spectral curve, define a set of symmetric multidifferentials $ \{ \omega_{g,n} \}_{g\geq 0, n\geq 1}$ on $ \Sigma^n$ via \emph{topological recursion} as follows: first, define the unstable cases by $ \omega_{0,1} \coloneqq y dx$ (this need only be defined locally near the $a_i$ using any primitive of $y$) and $ \omega_{0,2} \coloneqq B$. Then, for $ 2g-2+(n+1)>0$, the stable range, define
	\begin{equation}\label{TRrecursionequationLocal}
		\omega_{g,n+1} (z_{[n]}, z_{n+1}) \coloneqq \sum_{a \in R} \sum_{\{ 0\} \subsetneq I \subset \{ 0, \dotsc,r_a-1 \} } \Res_{z = a} \frac{\int_a^{z} \omega_{0,2}(\mathord{\cdot}, z_{n+1})}{\prod_{i=2}^m \big( \omega_{0,1} (z) - \omega_{0,1}(\sigma_a^i(z) )\big) } \mc{W}_{g,|I|+1,n} ( \sigma_a^{I} ( z ); z_{[n]} )\,,
	\end{equation}
	where $ \sigma_a$ is a generator of the local deck transformations of a primitive of $dx$ at $a$, and
	\begin{equation}
		\mc{W}_{g,m,n}(\zeta_{[m]};z_{[n]}) \coloneqq \sum'_{\substack{\mu \vdash [m] \\ \bigsqcup_{k=1}^{l(\mu )} N_k = [n] \\ \sum g_k=g+l( \mu )-n}} \prod_{k=1}^{l(\mu )} \omega_{g_k,|\mu_k|+|N_k|}(\zeta_{\mu_k},z_{N_k})
	\end{equation}
	where the prime on the summation means exclusion of any \( (g_k, |\mu_k| + |N_k|) = (0,1)\).
\end{definition}

\begin{remark}
	Often, the definition of spectral curves involves functions $ x$ and $ y$, in stead of their derivatives. However, these functions may not be defined globally on $x$, e.g. they may -- and in this paper will -- contain logarithmic terms. As most of the theory of topological recursion (with the notable exception of the global topological recursion of Bouchard-Eynard~\cite{BE13}) only depends on the derivatives, I have chosen to use this as a definition.
\end{remark}

\begin{theorem}[{\cite{BEO15,BS17}}]\label{ALE}
	Let $ \mc{C}$ be a rational spectral curve with simple zeroes of $dx$. A collection $ \{ \omega_{g,n} \}_{g\geq 0, n\geq 1}$ with $ \omega_{0,1} = y dx$ and $ \omega_{0,2} = B$ satisfies topological recursion if and only if the following hold:
	\begin{itemize}
		\item \textup{Meromorphicity:} For $ 2g-2+n > 0$, $ \omega_{g,n}$ extends to a meromorphic form on $ \Sigma^n$;
		\item \textup{Linear loop equation:} For any $ g$, $n$, and $a \in R$, 
			\begin{equation}
				\omega_{g,n+1}(z, z_{\llbracket n\rrbracket}) + \omega_{g,n+1}(\sigma_a(z), z_{\llbracket n \rrbracket})
			\end{equation}
			is holomorphic near $ z = a $ and has a simple zero at $ z = a$;
		\item \textup{Quadratic loop equation:} For any $ g$, $n$, and $a \in R$,
			\begin{equation}
				\omega_{g-1,n+2}(z, \sigma_a(z),z_{\llbracket n\rrbracket}) 
				+ \sum_{\substack{g_1 + g_2 = g\\I \sqcup J = \llbracket n\rrbracket}} \omega_{g_1,|I|_1+1}(z, z_I) \omega_{g_2,|J|+1}(\sigma_a(z),z_J)
			\end{equation}
			is holomorphic near $ z = a $ and has a double zero at $ z = a$;
		\item \textup{Projection property:} For $2g-2+n>0$, 
			\begin{equation}
				\omega_{g,n}(z_{\llbracket n\rrbracket}) = \sum_{a_1, \dotsc, a_n \in R} \Big( \prod_{j=1}^n \Res_{\zeta_j = a_j} \int_{a_j}^{z_j} \omega_{0,2}(z_j, \mathord{\cdot}) \Big) \omega_{g,n} (\zeta_{\llbracket n\rrbracket})\,.
			\end{equation}
	\end{itemize}
\end{theorem}

If only the meromorphicity and linear and quadratic loop equations hold, the problem is said to satisfy blobbed topological recursion, cf. \cite{BS17}. In this case, the $ \omega_{g,n}$ are determined by the spectral curve along with their holomorphic parts at ramification points.

One important reason to consider topological recursion is that the $\omega_{g,n}$ will often encode enumerative invariants in their Taylor series expansion around a given point of the spectral curve in a given coordinate. For us, this is also the case, as we consider the class given by the following theorem:

\begin{theorem}[\cite{BDKS20,BDKS20a}]\label{OScorrelators}
	Consider two formal power series
	\begin{equation}\label{psiandy}
		\hat{\psi} (\hbar^2,y) \coloneqq \sum_{k=1}^\infty \sum_{m=0}^\infty c_{k,m} y^k \hbar^{2m}\,, \qquad \hat{y}(\hbar^2,z) \coloneqq  \sum_{k=1}^\infty \hat{y}_k(\hbar^2) z^k \coloneqq \sum_{k=1}^\infty \sum_{m=0}^\infty s_{k,m}z^k\hbar^{2m}  \,,
	\end{equation}
	and their associated \emph{hypergeometric KP tau-function} or \emph{Orlov-Scherbin partition function}
	\begin{equation}\label{DefHypGeo}
		Z (\underline{p}) = e^{F(\underline{p})} = \sum_{\nu \in \mc{P}} \exp \Big( \sum_{\square \in \nu} \hat{\psi} (\hbar^2,-\hbar c_\square )\Big) s_\nu (\underline{p} ) s_\nu \big( \big\{ \frac{\hat{y}_k(\hbar^2)}{\hbar} \big\} \big)\,.
	\end{equation}
	Define
	\begin{equation}\label{XfromKP}
	\begin{alignedat}{3}
		\psi (y) 
		&
		\coloneqq \hat{\psi}(0,y)\,,  \qquad
		&
		y(z) 
		&
		\coloneqq \hat{y}(0,z)\,, \qquad
		&
		x(z) 
		&
		\coloneqq \log z - \psi(y(z))\,, 
		\\
		X(z) 
		&
		\coloneqq e^{x(z)}\,,
		&
		D 
		&
		\coloneqq \frac{\del}{\del x} \,, 
		&
		Q 
		&
		\coloneqq z \frac{d x}{d z}
	\end{alignedat}
	\end{equation}
	and write
	\begin{equation}\label{HnDef}
		H_n \coloneqq \sum_{k_1, \dotsc, k_n = 1}^\infty \frac{\del^n F}{\del p_{k_1} \dotsb \del p_{k_n}} \bigg|_{p=0} X_1^{k_1} \dotsb X_n^{k_n}\,.
	\end{equation}
	Then these can be decomposed as
	\begin{equation}\label{HnGenusDecomp}
		H_n = \sum_{g=0}^\infty \hbar^{2g-2+n} H_{g,n}\,,
	\end{equation}
	with $ H_{g,n} $ independent of $ \hbar$, and
	\begin{equation}\label{OSUnstable}
		DH_{0,1}(X(z)) = y(z)\,, \qquad H_{0,2}(X(z_1),X(z_2)) = \log \Big( \frac{z_1^{-1} - z_2^{-1}}{X_1^{-1} - X_2^{-1}} \Big)\,.
	\end{equation}
	If moreover $\frac{d\psi (y)}{dy} \big|_{y = y(z)}$ and $ \frac{d y(z)}{dz}$ have analytic continuations to meromorphic functions in $ z$ and all coefficients of positive powers of $\hbar^2$ in $ \hat{\psi}(\hbar^2,y(z))$ and $ \hat{y}(\hbar^2,z)$ are rational functions of $z$ whose singular points are disjoint from the zeroes of $ dx$, then the $n$-point differentials
	\begin{equation}
		\omega_{g,n} \coloneqq d_1 \dotsb d_n H_{g,n} + \delta_{g,0}\delta_{n,2} \frac{dX_1 \, dX_2}{(X_1-X_2)^2}
	\end{equation}
	can be extended analytically to $ (\P^1)^n$ as global rational forms, and the collection of $n$-point differentials satisfies meromorphicity and the linear and quadratic loop equations, i.e. blobbed topological recursion, for the curve $ (\P^1, dx(z), dy(z),B = \frac{dz_1\, dz_2}{(z_1-z_2)^2})$.\par
	Finally, if $ \hat{\psi} $ and $ \hat{y}$ belong to one of the two families
	\begin{align*}
		&
		\text{Family I}
		&
		\hat{\psi}(\hbar^2,y)
		&
		= \mc{S}(\hbar\del_y ) P_1(y) + \log \Big(\frac{P_2(y)}{P_3(y)}\Big)\,;
		&
		\hat{y}(\hbar^2, z) 
		&
		= \frac{R_1(z)}{R_2(z)}\,,
		\\
		&
		\text{Family II}
		&
		\hat{\psi}(\hbar^2,y)
		&
		= \alpha y\,;
		&
		\hat{y}(\hbar^2, z) 
		&
		= \frac{R_1(z)}{R_2(z)} +\mc{S} (\hbar z \del_z)^{-1} \log \Big(\frac{R_3(z)}{R_4(z)}\Big)\,,
	\end{align*}
	where $ \alpha \in \C^\times$ and the $ P_i$ and $R_j$ are arbitrary polynomials such that $ \psi(y)$ and $ y(z)$ are non-zero, but vanishing at zero, and no singular points of $y$ are mapped to branch points by $x$, then the $n$-point differentials also satisfy the projection property, and hence topological recursion, for the curve above.
\end{theorem}



\begin{remark}\label{ConstantTermpsi}
	It is possible to allow for constant terms in $ \hat{\psi}$ in \cref{psiandy}, but using quasihomogeneity of the $ s_\nu$ in \cref{DefHypGeo}, this can be absorbed in a recaling of the argument of $ \hat{y}$. From the spectral curve point of view, this follows from the fact that the two curves
	\begin{equation}
		\begin{cases} 
			X_1(z) 
			&
			= z e^{- \psi \circ y(z) + \log a} = az e^{-\psi \circ y (z)} 
			\\
			y_1 (z) 
			&
			= y(z) 
		\end{cases} 
		\qquad \text{and} \qquad 
		\begin{cases}
			X_2(z') 
			&
			= z' e^{-\psi \circ y(\frac{z'}{a})}
			\\
			y_2(z')
			&
			= y(\frac{z'}{a})
		\end{cases}
	\end{equation}
	can be identified via $ z' = az$.
\end{remark}

\begin{remark}
	We will consistently use the symbol $x$ for the function which is part of the spectral curve data and $ X$ for its exponential, which is the expansion parameter for this class of Hurwitz problems.
\end{remark}

We will need different parts of this theorem for the different parts of the proof. In particular, topological recursion is needed to obtain intersection numbers.

There is one small generalisation that can be made in \cref{OScorrelators} by using a homogeneity property, using that \cref{DefHypGeo} is a KP tau-function identically in $ \hbar$.

\begin{lemma} \label{TorusAction}
	There is a $\C^\times$ action on pairs $ (\hat{\psi}, \hat{y}) $ induced by rescaling of $\hbar$ in \cref{DefHypGeo} as follows:
	\begin{equation}
		\lambda \cdot \big(\hat{\psi} (\hbar^2, y), \hat{y}(\hbar^2, z) \big) = \big( \hat{\psi} (\lambda^{-2} \hbar^2,  \lambda^{-1}y),\lambda \hat{y} (\lambda^{-2} \hbar^2, z) \big)\,.
	\end{equation}
	This acts on $ H_{g,n}$ as $ \lambda \cdot H_{g,n} = \lambda^{2-2g-n}H_{g,n}$ and on the spectral curve by
	\begin{equation}
		\lambda \cdot (\P^1, dx(z), dy(z), B) = (\P^1, dx(z), \lambda dy(z), B)\,.
	\end{equation} 
	Hence it is compatible with the homogeneity of topological recursion of e.g. \cite[Section 4.1]{EO09}.
\end{lemma}

\begin{corollary} \label{ExtendedFamilyII}
	Family II of \cref{OScorrelators} can be extended to 
	\begin{equation}
		\hat{\psi}(\hbar^2,y)
		= \alpha y\,;
		\qquad
		\hat{y}(\hbar^2, z) 
		= \frac{R_1(z)}{R_2(z)} +\lambda \mc{S} (\lambda^{-1} \hbar z \del_z)^{-1} \log \Big(\frac{R_3(z)}{R_4(z)}\Big)\,,
	\end{equation}
	with the same conditions on $ \alpha $, the $ P_i$, and the $R_j$ as before, and $ \lambda \in \C^\times $.
\end{corollary}

\begin{proof}
	The constant $ \lambda $ has been absorbed in $ \alpha $ and the $ R_i$ where possible.
\end{proof}

A similar argument for Family I does not lead to an extended class, as this family is already invariant under the torus action.

\subsection{Topological recursion and cohomological field theories}\label{TRCohFT}

Topological recursion is strongly related to intersection theory of the moduli spaces of curves: there is a quite general correspondence between spectral curves and certain coherent collections of intersection classes in the moduli spaces. These coherent collections are cohomological field theories, which were orignally defined by Kontsevich and Manin \cite{KM94} to axiomatise Gromov-Witten theory.

\begin{definition}[\cite{KM94}]
	Let $ V$ be a vector space with a non-degenerate bilinear form $\eta$ and a distinguished vector $ \1$.
	A \emph{cohomological field theory with flat unit (CohFT)} on $(V,\eta, \1)$ is a collection of maps
	\begin{equation}
		\Omega_{g,n} \colon V^{\otimes n} \to H^*(\Mgn )\,,
	\end{equation}
	for all $ g \geq 0$, $ n \geq 1$ such that $ 2g-2+n >0$, such that
	\begin{itemize}
		\item $ \Omega_{g,n}$ is $ \mf{S}_n$-equivariant with respect to simultaneous permutation of the factors and the marked points;
		\item with respect to the glueing maps
		\begin{equation}
			\rho \colon \overline{M}_{g-1,n+2} \to \Mgn \,, \qquad \sigma \colon \overline{M}_{g,|I|+1} \times \overline{M}_{h,|J|+1} \to \overline{M}_{g+h, |I \sqcup J|}\,,
		\end{equation}
		we get
		\begin{align*}
			\rho^*\Omega_{g,n} (v_{\llbracket n \rrbracket}) 
			&
			= \Omega_{g-1,n+2} (v_{\llbracket n \rrbracket} \otimes \eta^\dagger )\,,
			\\
			\sigma^* \Omega_{g+h,|I|+|J|} (v_I \otimes v_J) 
			&
			= \Omega_{g, |I|+1} \otimes \Omega_{h,|J|+1} (v_I \otimes \eta^{\dagger} \otimes v_J)\,;
		\end{align*}
		\item With respect to the forgetful maps
		\begin{equation}
			\pi \colon \overline{M}_{g,n+1} \to \Mgn\,,
		\end{equation}
		we have
		\begin{equation}
			\pi^* \Omega_{g,n} (v_{\llbracket n\rrbracket}) = \Omega_{g,n+1} (v_{\llbracket n\rrbracket} \otimes \1)\,.
		\end{equation}
	\end{itemize}
\end{definition}

There is a large group acting on the space of CohFTs, called the Givental group \cite{Giv01,Sha09,Tel12}. It consists of $ R(u) \in \Id + u \End (V) \llbracket u \rrbracket $ such that $ R(u) R^{\dagger} (-u) = \Id$. It is called the unit-preserving action in case one also considers CohFTs without unit.

\begin{theorem}[\cite{Eyn14,DOSS14,BKS20}]\label{EynardDOSS}
	Consider a compact rational spectral curve $ (\P^1, dx,dy, B)$, and define $ V^*$ to be the space of residueless meromorphic one-forms on $\P^1$ with poles only at $a \in R$ of order at most $ r_a +1$. Choose a basis $ \{ d\xi^j\}_{j \in J} $ of $V^*$ with dual basis $ e_j$ and define $ d\xi^j_k = (d \circ\frac{1}{dx})^k d\xi^j$. Then
	\begin{equation}\label{CorrelatorsAsIntersection}
		\omega_{g,n}(z_1,\dotsc, z_n) 
		=
		\sum_{j_1, \dotsc, j_n \in J} \int_{\Mgn} \Omega_{g,n} (e_{j_1} \otimes \dotsb \otimes e_{j_n}) \prod_{i=1}^n \sum_{k_i=0}^\infty \psi_i^{k_i} d\xi_{k_i}^{j_i} (z_i)\,,
	\end{equation}
	where $ \Omega$ is a cohomological field theory on $V$, given explicitly by acting on a direct sum of Witten $r_a$-spin classes for all ramification points of order $r_a$ by an $R$ determined by the spectral curve.
\end{theorem}

\begin{remark}
	The results of \cite{Eyn14,BKS20} do not mention CohFTs, but rather give a relation between local spectral curves and intersection numbers. In order to obtain a CohFT, a condition is required, cf. \cite[Equation (17)]{LPSZ-Chiodo}. As noted in \cite[Section 2.6]{DNOPS19}, in case the spectral curve is compact and $dx $ is meromorphic with simple zeroes, this condition is satisfied by \cite[Appendix B]{Eyn14}. In case of higher order zeroes, the same holds, using \cite[Section 7.2.3]{BKS20}.
\end{remark}

\begin{remark}
	The space $ V^* $ is naturally related to the projection property of \cref{ALE}: the $ d\xi_k^j$ span the image of the projection operator. Its dimension, the rank of the CohFT, equals the degree of the divisor of zeroes of $dx$.\par
	There are two common choices for the basis  $ d\xi^j$, depending on a local coordinate $ \zeta_a$ around a ramification point $a$ such that $ x(z) = \zeta_a(z)^{r_a} + x(a)$. One is $ d\xi^{a,k}(z) = \Res_{z' = a} \Big( \int_a^{z'} B(z,\cdot) \Big) \frac{d\zeta (z')}{\zeta (z')^k}$, with $ 1 \leq k \leq r_a-1$, cf. \cite[Equation~(80)]{BKS20}, while the other is $ \xi^a(z) = \int^z\frac{B(\zeta_a, \cdot )}{d\zeta_a}\big|_{\zeta_a = 0}$, in case $ r_a = 2$, cf. \cite[Equation~(2.23)]{GKL21}. Both have merit, depending on the situation, but they are not compatible.\par
	Furthermore, several normalisation conventions exist for the recursion operator linking $ d\xi^j_k$ to $ d\xi^j_{k+1}$. These different conventions can be related by rescaling $ \Omega$ and the correlators, using that the integrand must be of degree $ 3g-3+n$.
\end{remark}

So the $\omega_{g,n}$ we are concerned with can be expanded in different ways: as a formal series around $ X = 0$ by \cref{OScorrelators}, and on a basis of meromorphic differentials with poles at the zeroes of $dx$ by \cref{EynardDOSS}. The change of variables we require is found by relating these different expansions.\par
In order to apply the Eynard-DOSS correspondence to get a good change of variables, we will want to take a different basis of $ V^*$. It turns out to be useful to relate to powers of our preferred coordinate $z$, so the basis we take is $ \xi^j \coloneqq \big( \frac{dx}{dz}\big)^{-1}z^j= \frac{d}{dx} \frac{z^{j+1}}{j+1}$.

\begin{definition}
	Let $ \Omega $ be a cohomological field theory on a space $(V, \eta)$ with a basis $ \{ e_j\}_{j \in J}$. Its \emph{generating function} $ G_\Omega $  is defined as
	\begin{equation}
		G_{\Omega}(\{ T^j_k \mid k \geq 0, j \in J \} ) \coloneqq \sum_{\substack{g,n\\2g-2+n>0}} \frac{\hbar^{2g-2+n}}{n!} \sum_{\substack{ j_i, \dotsc, j_n \in J}} \int_{\Mgn} \Omega (e_{j_1} \otimes \dotsb \otimes e_{j_n})\prod_{i=1}^n \sum_{k_i=1}^\infty \psi_i^{k_i} T^{j_i}_{k_i}\,,
	\end{equation}
	where we write $ \{ T^{j}\} $ for the basis of $ V^*$ dual to $ \{ e_{j}\}$ and $ T^j_k = T^j \otimes p_k$.
\end{definition}

The main theorem of this paper is the following:

\begin{theorem}\label{KPforHurwitzCohFT}
	If a cohomological field theory $ \Omega$ is obtained from \cref{EynardDOSS} applied to either family in \cref{OScorrelators}, then the exponential of $G_\Omega (\underline{T}(\underline{q})) $ is a KP tau-function in $\{ t_d = \frac{q_d}{d} \}$, where the $T^j_k(\underline{q})$ are defined by
	\begin{equation}
		T_{-1}^j = \frac{1}{j+1} q_{j+1}\,,
		\qquad
		T_{k+1}^j = \sum_{m=1}^\infty \sum_{l=0}^\infty m \mc{T}_l q_{m+l} \frac{\del}{\del q_m} T_k^j\,,
		\quad 
		\text{with $\mc{T}_l$ given by}
		\quad
		Q(z)^{-1} = \sum_{l=0}^\infty \mc{T}_l z^l\,.
	\end{equation}
\end{theorem}

The proof of this theorem is given in \cref{ProofKPForBDKSIntNumbers}.

\begin{remark}
	The proof of this theorem does not use anything specific to the families mentioned, it just requires topological recursion to obtain a cohomological field theory. As soon as topological recursion is proved for another hypergeometric tau-function and the spectral curve fits in the scope of \cref{EynardDOSS}, this theorem generalises. 
\end{remark}

The first two KP equations in $q$ variables are

\begin{align}
	0 &= \frac{\del^2 F}{\del q_2^2} - \frac{\del^2 F}{\del q_3 \del q_1}  + \frac{1}{12} \frac{\del^4 F}{\del q_1^4} + \frac{1}{2} \Big( \frac{\del^2 F}{\del q_1^2} \Big)^2\,;
	\label{KP1}
	\\
	0 &= \frac{\del^2 F}{\del q_3 \del q_2} - \frac{\del^2 F}{\del q_4 \del q_1} + \frac{1}{6} \frac{\del^4 F}{\del q_2 \del q_1^3} + \frac{\del^2 F}{\del q_2 \del q_1} \frac{\del^2 F}{\del q_1^2}\,.
\end{align}

\begin{example}[Naive single Hodge]\label{NaiveHodge}
	Let us consider the functions
	\begin{equation}
		\hat{\psi} (\hbar^2, y) = y \,, \qquad \hat{y}(\hbar^2, z) = z\,.
	\end{equation}
	Then we find
	\begin{equation}
		x(z) = \log z - z\,, \qquad X(z) = z \, e^{-z} \,, \qquad Q = 1-z\,.
	\end{equation}
	This is the usual shape of spectral curve for simple Hurwitz numbers \cite{BM08,BEMS11} , so the CohFT associated to it by \cref{EynardDOSS} is the single Hodge class $ \Lambda (-1)$ via the ELSV formula \cite{ELSV}. In this case, writing $ T_k = \sum_{m=1}^\infty c_{k,m} q_m$, \cref{KPforHurwitzCohFT} yields
	\begin{equation}
		c_{k+1,m} = \sum_{j=0}^\infty j c_{k,j}\,.
	\end{equation}
	Along with the initial condition $ c_{-1,m} = \delta_{m,1}$, this shows that $ c_{k,m} = \Stirling{k+m}{m}$ for $ m > -1$, the Stirling numbers of the second kind. In particular,
	\begin{equation}
	\begin{aligned}
		T_0 
		&= q_1 + q_2 + q_3 + q_4 + q_5 + ...\,,
		\\
		T_1 
		&= q_1 + 3 q_2 + 6 q_3 + 10 q_4 + 15 q_5 + ...\,,
		\\
		T_2 
		&= q_1 + 7 q_2 + 25 q_3 + 65 q_4 + 140 q_5 + ...
	\end{aligned}
	\end{equation}
	Note that these are infinite sums, in contrast to the ones Kazarian found in \cite{KazarianHodge}, cf. \cref{KazarianRecursion,KazarianTDef}, even though both are related to single Hodge integrals. This phenomenon is explained by the arbitrary choice of a rational parametrisation of the spectral curve, formalised in \cref{XMoebius}.\par
	Using the intersection numbers
	\begin{align}
		\int_{\bar{\M}_{0,3}} 1 &=  \int_{\bar{\M}_{0,4}} \psi_i = 1\,,\\
		\int_{\bar{\M}_{1,1}} \lambda_1 &= \int_{\bar{\M}_{1,1}} \psi_1 = \int_{\bar{\M}_{1,2}} \psi_i^2 = \int_{\bar{\M}_{1,2}} \psi_1 \psi_2 = \int_{\bar{\M}_{1,2}} \lambda_1 \psi_1 =  \frac{1}{24}\,,
	\end{align}
	we see that
	\begin{equation}
		G_{\Lambda (-1)}(\underline{T}) = \hbar \big( \frac{1}{6} T_0^3 + \frac{1}{24} T_1 - \frac{1}{24} T_0\big) + \hbar^2 \big( \frac{1}{6} T_0^3 T_1 + \frac{1}{48} T_1^2 + \frac{1}{24} T_0T_2 - \frac{1}{24} T_0T_1 \big) + \mc{O}(\hbar^3)\,.
	\end{equation}
	From this, we obtain
	\begin{equation}
	\begin{aligned}
		\frac{\del^2 G_{\Lambda (-1)}(\underline{T}(\underline{q}))}{\del q_2^2} 
		&= 
		\hbar \sum_{k>0} q_k + \hbar^2 \bigg( \sum_{k,l >0} \Big( \frac{l(l+1)}{2} + 3\Big) q_k q_l + \frac{2\cdot 3^2}{48} + \frac{2\cdot 7}{24} - \frac{2\cdot 3}{24} \bigg) + \mc{O}(\hbar^3)\,,
		\\
		\frac{\del^2 G_{\Lambda (-1)}(\underline{T}(\underline{q}))}{\del q_1 q_3} 
		&=
		\hbar \sum_{k>0} q_k + \hbar^2 \bigg( \sum_{k,l >0} \Big( \frac{l(l+1)}{2} + \frac{7}{2} \Big) q_k q_l + \frac{2\cdot 6}{48} + \frac{25 + 1}{24} - \frac{6 +1}{24} \bigg) + \mc{O}(\hbar^3)\,,
		\\
		\frac{\del^4 G_{\Lambda (-1)}(\underline{T}(\underline{q}))}{\del q_1^4}
		&=
		\hbar^2 \frac{24}{6} + \mc{O}(\hbar^3)\,,
		\\
		\frac{\del^2 G_{\Lambda (-1)}(\underline{T}(\underline{q}))}{\del q_1^2}   
		&=
		\hbar \sum_{k>0} q_k + \mc{O}(\hbar^2)\,,
	\end{aligned}
	\end{equation}
	which does show that $ G_{\Lambda (-1)}(\underline{T}(\underline{q}))$ solves \cref{KP1} up to second order in $\hbar$.\par
	I would like to thank P. Norbury for using this example to check the results of this paper.
\end{example}

\subsection{The Mari\~{n}o-Vafa formula and KP for topological vertex amplitudes}\label{MVandKP}

A particularly interesting family of hypergeometric tau-functions is given by the theory of the topological vertex, or triple Hodge integrals. For the triple Hodge integrals, the ELSV-type formula required is the Mariño-Vafa formula~\cite{MarinoVafa}. 
This theory is the particular case for $\C^3$ of the Gromov-Witten theory of toric Calabi-Yau threefolds, conjectured by Bouchard-Klemm-Mariño-Pasquetti~\cite{BKMP09} to satisfy topological recursion. The case we are interested in was proved in \cite{Che18,Z09}, as well as in \cite{Eyn11} as an example of the general correspondence of \cref{EynardDOSS}, while the general BKMP conjecture was proved in \cite{EyOr}.\par
In this section, we use the triple Hodge integrals as an example of our general theory, using methods slightly adapted to this special case. We will see in \cref{Finiteness} why this case is particularly nice.


\begin{definition}
	The \emph{triple Hodge cohomological field theory with Calabi-Yau condition} is the one-dimensional CohFT $ \mathop{\textup{TH}}_{g,n}(a,b,c) = \Lambda (a) \Lambda (b) \Lambda (c)$, where the parameters $ a, b, c$ satisfy $ \frac{1}{a} + \frac{1}{b} + \frac{1}{c} = 0$.\par
	We write 
	\begin{equation}\
		G_{\textup{TH}}(a,b,c;\underline{T} ) \coloneqq  G_{\mathop{\textup{TH}}(a,b,c)}(\underline{T})\,.
	\end{equation}
\end{definition}

An adapted application of \cref{KPforHurwitzCohFT} is given in the following theorem. This theorem has already been proved by Alexandrov~\cite{Ale20}, here we give a new proof.

\begin{theorem}[{\cite[Theorem 2]{Ale20}}]\label{thm:KPforTH}
	Define $ T_0(\underline{q}) \coloneqq q_1$, $ T_{k+1}(\underline{q}) \coloneqq \sum_{m=1}^\infty m (u^2 q_m + u \frac{w +2}{\sqrt{w + 1}} q_{m+1} + q_{m+2})\frac{\del}{\del q_m} T_k $. Then
	\begin{equation}
		G_{\textup{TH}}\big( -u^2 ,-u^2 w ,\frac{u^2 w}{w +1};\{ T_{k}(\underline{q}) \} \big)
	\end{equation}
	is a solution of the KP hierarchy with respect to the variables $ \{ t_d = \frac{q_d}{d} \} $, identically in $ u $ and $ w $.
\end{theorem}

In this particular case, we will make slightly different choices to end up with the formulation above.

\begin{remark}
	Note that the triple $ a= -u^2 $, $ b= -u^2w $, $ c = \frac{u^2 w}{w +1} $ does indeed satisfy $ \frac{1}{a} + \frac{1}{b} +\frac{1}{c} = 0$, and moreover any triple satisfying this condition can be written this way.
\end{remark}

\begin{remark}
	In the limit $ w \to 0$, this theorem reduces to the main theorem, 2.1, of \cite{KazarianHodge}. In the limit $ u \to 0$, it reduces to the Witten-Kontsevich theorem~\cite{Witten,Kontsevich}: in that limit $ T_d\to (2d-1)!! q_{2d+1}  $ and independence of even parameters reduces the KP hierarchy to the KdV hierarchy.
\end{remark}

Before giving the Mari\~{n}o-Vafa formula, note that in genus zero
\begin{equation}
	\int_{\overline{\mc{M}}_{0,n}} \frac{\Lambda (a) \Lambda (b) \Lambda (c)}{ \prod_{i=1}^n 1-\mu_i \psi_i^{d_i}} = |\mu |^{n-3}
\end{equation}
for $ n \geq 3$, and this serves as a definition for $ n = 1,2$. These terms are not included in $ G_{\textup{TH}} $.

\begin{theorem}[Mari\~{n}o-Vafa formula, \cite{MarinoVafa,LLZ03,OP04}]\label{MVthm}
	There is a relation between triple Hodge integrals and characters of symmetric groups, as follows:
	\begin{equation}
	\begin{split}
		\exp \bigg( \sum_\mu \sum_{g = 0}^\infty \frac{(w + 1)^{g +n-1}}{|\Aut \mu |} \prod_{i=1}^n \frac{\prod_{j=1}^{\mu_i-1} (\mu_i + j w )}{(\mu_i-1) !} &\int_{\Mgn} \!\! \frac{\Lambda (-1) \Lambda (-w ) \Lambda \big(\frac{w}{w + 1}\big)}{\prod_{i=1}^n (1- \mu_i \psi_i )} \, \beta^{2g-2+n + | \mu |} \, p_\mu \bigg) 
		\\
		&
		= \sum_{m=0}^\infty \sum_{\mu, \nu \vdash m} \frac{\chi^\nu_\mu}{z_\mu} e^{(1+\frac{w}{2})\beta f_2(\nu )} \prod_{\square \in \nu } \frac{\beta w}{\varsigma (\beta w h_\square )} p_\mu\,. \label{MV-formula}
	\end{split}
	\end{equation}
	On the right-hand side the sum is over all partitions $ \nu $ of size equal to $ |\mu | $, the product is over all boxes in the Young diagram of $ \nu $, and $ h_\square $ is the hook length of the box  $ \square $. Furthermore, $ f_2 (\nu) = \frac{1}{2}\sum_{j}  (\nu_j-j+\tfrac{1}{2})^2 - (-j + \tfrac{1}{2})^2$ is the shifted symmetric sum of squares.
\end{theorem}

\begin{remark}
	Even though it seems the triple Hodge class in this formula only depends on one parameter, $ w $, the parameter $ \beta $ can be interpreted in this way as well, entering as a cohomological grading parameter. Hence, the formula does govern the entire generating function of triple Hodge integrals.\par
In the limit $ w \to 0 $, the Mari\~{n}o-Vafa formula reduces to the ELSV formula, as the product over boxes simplifies to the hook length formula for the dimension of the $ \mf{S}_{|\mu |} $-representation associated to $ \nu $.
\end{remark}

\begin{remark}
	This formula is perfectly well-behaved for $ w = -1$, but \cref{thm:KPforTH} does not make sense in this case. From the general \cref{KPforHurwitzCohFT}, we will see that in this case $X$ is a Möbius transformation, and hence conforms to \cref{XMoebius}.\par
	By symmetry in the arguments of the $ \Lambda $-classes, the point $ w= -1$ is equivalent to the limit $ w \to \infty $, which in the conventional formulation of the Mari\~no-Vafa formula is the initial condition for the cut-and-join equation used to prove the formula, see \cite[Theorem 3.3]{Z03}. In this case, the integral reduces to $ \int_{\overline{\mc{M}}_{g,1}} \lambda_g \psi^{2g-2} $ by Mumford's relation. These integrals were calculated by Faber and Pandharipande \cite{FaPa00}.
\end{remark}

The right-hand side of the Mariño-Vafa formula is a hypergeometric KP tau-function, which can be seen explicitly by the following lemma. In essence this lemma was used by both \cite{LLZ03,OP04} to prove the Mariño-Vafa formula.

\begin{lemma}\label{MarinoVafaIsBDKS}
	Relabel parameters in \cref{MV-formula} by $ \beta \to \frac{\hbar}{w} $ and $ p_k \to \big( \frac{\beta w}{\hbar}\big)^k p_k$ to obtain
	\begin{equation}
	\begin{split}
		\exp \bigg( \sum_\mu \sum_{g = 0}^\infty \frac{(w + 1)^{g +n-1}}{|\Aut \mu |} \prod_{i=1}^n \frac{\prod_{j=1}^{\mu_i-1} (\mu_i + j w )}{(\mu_i-1) !} \int_{\Mgn} \!\! &\frac{\Lambda (-1) \Lambda (-w ) \Lambda \big(\frac{w}{w + 1}\big)}{\prod_{i=1}^n (1- \mu_i \psi_i )} \, \hbar^{2g-2+n} \beta^{| \mu |} w^{|\mu | + 2 - 2g -n} \, p_\mu \bigg) 
		\\
		&
		= \sum_{m=0}^\infty \sum_{\mu, \nu \vdash m} \frac{\chi^\nu_\mu}{z_\mu} e^{(\frac{1}{w}+\frac{1}{2})\hbar f_2(\nu )} \prod_{\square \in \nu } \frac{\beta w}{\varsigma (\hbar h_\square )} p_\mu\,. 
		\end{split}
	\end{equation}
	This right-hand side may alternatively be written as a hypergeometric KP tau-function in the shape of \cref{OScorrelators}, with
	\begin{equation*}
		\hat{\psi} (\hbar^2,y) = -\frac{y}{w} \,, 
		\qquad 
		\hat{y}(\hbar^2,z) = \sum_{k=1}^\infty \frac{1}{k \mc{S} (\hbar k)} (\beta wz)^k\, ,
		\qquad
		X (z) = z (1-\beta wz)^{1/w}\,.
	\end{equation*}
	This fits in Family II.
\end{lemma}

\begin{proof}
	By basic theory of symmetric functions, $ \sum_{\mu \vdash m} \frac{\chi^\nu_\mu}{z_\mu} p_\mu = s_\nu(\underline{p})$. Also, by \cite[Equations (0.6), (0.7)]{OP04},
	\begin{equation*}
		 \frac{1}{\prod_{\square \in \nu} q^{h_\square/2} - q^{-h_\square/2}} = q^{-|\nu|/2 - f_2(\nu)/2} s_\nu(1,q^{-1},q^{-2},\dotsc )
	\end{equation*}
	where here the $ q^{-k}$ are the `usual' variables, i.e. the ones in which $s_\nu$ is symmetric, not the power sum variables.\par
	Writing $q = e^{\hbar}$ and using that $ f_2(\nu) =   \sum_{\square \in \nu} c_\square $ gives
	\begin{align*}
		\sum_{m=0}^\infty \sum_{\mu, \nu \vdash m} \frac{\chi^\nu_\mu}{z_\mu} e^{(\frac{1}{w}+\frac{1}{2})\hbar f_2(\nu )} \prod_{\square \in \nu } \frac{\beta w}{\varsigma (\hbar h_\square )} p_\mu 
		&=
		\sum_{m=0}^\infty \sum_{\nu \vdash m} s_\nu(\underline{p}) e^{\sum_{\square \in \nu}\frac{\hbar}{w} c_\square} s_\nu \Big( \Big\{ \frac{\beta w}{e^{\hbar (l+\frac{1}{2})}} \Big\}_{l=0}^\infty \Big) \,.
	\end{align*}
	To revert to power-sum variables, we use that $ p_l \Big( \Big\{ \frac{\beta w}{e^{\hbar (l+\frac{1}{2})}} \Big\}_{l=0}^\infty \Big) = \frac{(\beta w)^k}{\varsigma (\hbar k)}$, and inserting this in the definition of $\hat{y}$ yields the result.
\end{proof}

If one were to relabel in stead by $ \beta \to \beta \hbar $ and $ p_k \to \hbar^{-k} p_k$, i.e. just to naively introduce a parameter $ \hbar$, one would not end up in Family II, but in the extension of \cref{ExtendedFamilyII}. These two choices are related by \cref{TorusAction}, for $ \lambda = \beta w$.

Zhou~\cite{Z10} also explored this relation between triple Hodge integrals and integrable hierarchies, extending it to the 2-Toda hierarchy and to certain relative Gromov-Witten theories.

\section{KP hierarchy for intersection numbers}\label{sec:KPforTH}

In this section, we will formulate and prove the main theorem, generalising Kazarian's method to the generating functions of intersection numbers coming from hypergeometric tau-functions.

\subsection{The change of variables}\label{changeofvars}

We will interpret any $X(z)$ defined by \cref{psiandy,XfromKP} as giving a change of coordinates. For this, define a linear correspondence $ \Theta $ between power series in $ X$ or $ z$ on the one hand and linear series in $ p$ or $ q$ on the other by
\begin{align}
	p_k 
	&\leftrightarrow X^k \,, 
	& q_m 
	&\leftrightarrow z^m\,.
\end{align}

This defines a change of coordinates as follows:

\begin{definition}\label{LinearCorrespondence}
	We define a linear morphism between power series in $ \{ p_m \}_{m \geq 1} $ and $ \{ q_d \}_{d \geq 1} $ by
	\begin{equation}\label{ptoq}
		p_k (\underline{q}) = \sum_{m = k}^\infty c_k^m q_m 
		\qquad
		\text{with $ c_k^m $ given by}
		\qquad
		X^k = \sum_{m=k}^\infty c_k^m z^m\,.
	\end{equation}
\end{definition}

In order to make this change of coordinates and remain within the realm of solutions of the KP hierarchy, we should flow along the action of the infinite general linear algebra. Hence, we should find the infinitesimal flow associated to this change. For this, we introduce a flow parameter $\beta $ by
\begin{equation}
	X_\beta (z) \coloneqq \frac{1}{\beta} X(\beta z) = z e^{-\psi (y (\beta z))}\,,
\end{equation}
such that $ X_0(z) = z$ and $ X_1(z) = X(z)$.

\begin{lemma}\label{FlowXalongbeta}
	For $ X_\beta(z) \coloneqq \frac{1}{\beta} X(\beta z) $, where $X(z) = z + \mc{O}(z^2)$, and with $Q(z) \coloneqq \frac{z}{X} \frac{dX}{dz}$, the flow along $\beta$ of the function $ X_\beta $ is given by
	\begin{equation}
		\frac{\del X_\beta}{\del \beta} = \Big( 1- \frac{1}{Q(\beta z)} \Big) \frac{z}{\beta} \frac{\del X_\beta}{\del z} = \frac{1}{\beta} \big( Q(\beta z) -1\big) X_\beta\,.
	\end{equation}
\end{lemma}
\begin{proof}
	By definition of $Q$, $ X = \frac{z}{Q(z)} \frac{d X}{d z}$. Therefore,
	\begin{align*}
		\frac{\del X_\beta}{\del \beta}
		&=
		\frac{\del}{\del \beta} \Big( \frac{1}{\beta} X(\beta z)\Big)
		=
		\frac{z}{\beta} \frac{d X}{dz} \Big|_{z \to \beta z} - \frac{1}{\beta^2} X (\beta z)
		\\
		&=
		\frac{z}{\beta} \frac{d X}{dz} \Big|_{z \to \beta z} - \frac{1}{\beta^2} \Big( \frac{z}{Q(z)} \frac{d X}{dz} \Big) \Big|_{z \to \beta z}
		=
		\Big( 1 - \frac{1}{Q(\beta z)} \Big) \frac{z}{\beta} \frac{\del X_\beta}{\del z}\,.\qedhere
	\end{align*}
\end{proof}

We will use this with \cite[Theorem 2.5]{KazarianHodge}, which uses the $ \widehat{\mf{gl}(\infty)} $ action on $ \tau $-functions:

\begin{theorem}[\cite{KazarianHodge}] \label{changeofvarsforKP}
	In the situation of a correspondence like \cref{ptoq}, there is a quadratic function $ Q(\underline{p} )$ such that the transformation sending an arbitrary series $ \Phi (\underline{p} )$ to the series $ \Psi (\underline{q} )=(\Phi +Q)\big|_{p \to p(\underline{q})} $ is an automorphism of the KP hierarchy: it sends solutions to solutions.
\end{theorem}

The function $Q(\underline{p})$ is not unique.

\begin{proposition}\label{KPpreservingBDKS}
	In the general situation of \cref{OScorrelators}, without analytic assumptions, the quadratic function for the change of variables of \cref{LinearCorrespondence} can be taken to be $- \frac{1}{2} \Theta (H_{0,2})$.
\end{proposition}

\begin{proof}
	Consider the more general linear correspondence $ \Theta_\beta $ between power series in $ X_\beta $ or $ z$ on the one hand and linear series in $ p$ or $ q$ on the other by
	\begin{align*}
		p_k 
		&\leftrightarrow X_\beta(z)^k \,, 
		& q_m 
		&\leftrightarrow z^m\,.
	\end{align*}
	This gives a linear morphism between power series in $ \{ p_m \}_{m \geq 1} $ and $ \{ q_d \}_{d \geq 1} $ by
	\begin{align*}
		p_k (\beta; \underline{q}) 
		&= \sum_{m = k}^\infty c_k^m q_m 
		&\text{with $ c_k^m $ given by}
		&& X_\beta^k 
		&= \sum_{m=k}^\infty c_k^m z^m\,,
	\end{align*}
	such that $ p_k(0; \underline{q}) = q_k$.\par
	Under $\Theta_\beta$, the operator $ z^{m+1} \frac{\del}{\del z}$ transforms into $ \sum_{k =1}^\infty k q_{m+k} \frac{\del}{\del q_k}$, which is the differential part of $L_m(\underline{q})$. The polynomial part of this operator is
	\begin{equation*}
		\frac{1}{2} \sum_{k=1}^{m-1} q_k q_{m-k}\,,
	\end{equation*}
	which under the correspondence transforms into
	\begin{equation*}
		\frac{1}{2} \sum_{k=1}^{m-1} z_1^{k} z_2^{m-k} = \frac{1}{2} z_1z_2 \frac{z_1^{m-1} - z_2^{m-1}}{z_1-z_2} = - \frac{1}{2} \frac{z_1^{m-1} -z_2^{m-1}}{z_1^{-1} -z_2^{-1}}\,.
	\end{equation*}
	Therefore, the correction to be made to \cref{FlowXalongbeta} to obtain a KP-preserving flow is found by the substitution $ f(z) z \frac{\del}{\del z} \to - \frac{1}{2} \frac{1}{z_1^{-1} - z_2^{-1}} \big( z_1^{-1} f(z_1) - z_2^{-1} f(z_2) \big) $ for a series $ f(z) \in z \C \llbracket z \rrbracket$. Note that $ \frac{1}{\beta} \Big( 1- \frac{1}{Q(\beta z)}\Big)$ satisfies these requirements, and we find that the differential operator of \cref{FlowXalongbeta} needs to be completed by
	\begin{equation*}
		-\frac{1}{2\beta (z_1^{-1} - z_2^{-1})} \bigg( z_1^{-1} \big( 1- \frac{1}{Q(\beta z_1)} \big) - z_2^{-1} \big( 1 - \frac{1}{Q(\beta z_2)} \big) \bigg) 
		=
		\frac{1}{2\beta(z_1^{-1} -z_2^{-1})} \bigg( \frac{1}{z_1 Q(\beta z_1)} - \frac{1}{z_2 Q(\beta z_2)}\bigg) -\frac{1}{2\beta}\,.
	\end{equation*}
	By a similar calculation as for \cref{FlowXalongbeta},
	\begin{equation*}
		\frac{\del z}{\del \beta}\Big|_{X \text{ const.}}  = \frac{1}{\beta} \Big(\frac{1}{Q(\beta z)} - 1\Big) z\,,
	\end{equation*}
	from which it follows that
	\begin{align*}
		-\frac{\del H_{0,2}}{\del \beta} \Big|_{X \text{ const.}} 
		&= 
		-\frac{\del}{\del \beta} \log \Big( \frac{z_1^{-1} - z_2^{-1}}{X_1^{-1} - X_2^{-1}}\Big)\Big|_{X \text{ const.}} 
		\\
		&=
		\frac{1}{z_1^{-1} - z_2^{-1}} \Big( z_1^{-2} \frac{\del z_1}{\del \beta}\Big|_{X \text{ const.}}  - z_2^{-2} \frac{\del z_2}{\del \beta}\Big|_{X \text{ const.}}  \Big)
		\\
		&= \frac{1}{\beta (z_1^{-1} -z_2^{-1})} \Big( \frac{1}{z_1 Q(\beta z_1)} - \frac{1}{z_2 Q(\beta z_2)} \Big) - \frac{1}{\beta}\,,
	\end{align*}
	which, up to a factor $2$, is exactly the polynomial correction needed.\par
	From these calculations, we find that 
	\begin{equation*}
		A \coloneqq \Big( 1- \frac{1}{Q(\beta z)} \Big) \frac{z}{\beta} \frac{\del}{\del z} -\frac{1}{2}\frac{\del H_{0,2}}{\del \beta} \Big|_{X \text{ const.}} 
	\end{equation*}
	corresponds to a linear combination of $ L_m$ under $\Theta_\beta$, and hence preserves KP.
	Now consider a KP tau-function $ \Phi (\underline{p})$ and define the function $ Z(\beta, \underline{q}) = \exp\big( \Phi (\underline{p}(\beta, \underline{q}) - \frac{1}{2} \Theta (H_{0,2}) )\big)$. Then
	\begin{align}
		\frac{\del}{\del \beta} Z
		&=
		\Big( \sum_{k=1}^\infty \frac{\del p_k(\beta, \underline{q})}{\del \beta} \frac{\del}{\del p_k} - \frac{1}{2} \Theta \Big(\frac{\del H_{0,2}}{\del \beta} \Big|_{X \text{ const.}}  \Big) \Big) Z
		\\
		&=
		\Theta\Big( \Big( 1- \frac{1}{Q(\beta z)} \Big) \frac{z}{\beta} \frac{\del}{\del z} - \frac{1}{2} \frac{\del H_{0,2}}{\del \beta} \Big|_{X \text{ const.}}  \Big) Z(\beta)
		\\
		&=
		\Theta( A )\, Z (\beta)
	\end{align}
	As $ Z(0) = Z$, and $ \Theta (A) $ preserves $ \tau $-functions of KP, this automorphism does indeed preserve solutions.
\end{proof}

\begin{corollary}\label{KPtauWithoutUnstable}
	For $Z(\underline{p})$ defined by \cref{DefHypGeo}, $ Z(\underline{p}) \exp \big(- \Theta ( \hbar^{-1}H_{0,1} + \frac{1}{2} H_{0,2})\big) \big|_{p \to p(\underline{q})} $ is also a KP tau-function, whose logarithm does not contain unstable terms.
\end{corollary}

\begin{proof}
	As all equations in the KP hierarchy only contain at least second derivatives of $F$, addition of a linear term $- \Theta ( \hbar^{-1} H_{0,1}) $ preserves solutions. By \cref{KPpreservingBDKS}, subtracting the $ (0,2)$ term and changing $ p \mapsto p(\underline{q})$ is an automorphism as well.
\end{proof}

\begin{corollary}\label{XMoebius}
	In case $X(z)$ is a Möbius transformation with the shape of \cref{XfromKP}, i.e. $ X(z) = \frac{az}{1+ bz}$ (taking into account \cref{ConstantTermpsi}), this quadratic function can be taken to be $0$.
\end{corollary}
\begin{proof}
	By direct calculation,
	\begin{equation*}
		H_{0,2}(z_1,z_2) = \log \Big( \frac{z_1^{-1} - z_2^{-1}}{X(z_1)^{-1} - X(z_2)^{-1}} \Big) =\log a \,.
	\end{equation*}
	Comparing this with the proof of \cref{KPpreservingBDKS}, the quadratic correction is needed to complete the operator $A$, which only depends on $ \frac{\del H_{0,2}}{\del \beta}$. As this vanishes in the present case, we may as well omit the entire correction.
\end{proof}

\begin{remark}
	The usual $B$-function of topological recursion,
	\begin{equation}
		B(z_1,z_2) = \frac{dz_1 \, dz_2}{(z_1-z_2)^2} = d_1 d_2 \log (z_1^{-1} - z_2^{-1})\,,
	\end{equation}
	is invariant under all Möbius transformations, so $ d_1 d_2 H_{0,2}$ vanishes if $X$ is any Möbius transformation. However, this is not the case for $H_{0,2}$ itself: it is invariant under a one-dimensional subgroup, changes by a constant under the two-dimensional subgroup above, but under other Möbius transformations also changes by addition of terms like $ \log z_i$.\par
	Viewed another way, these more general Möbius transformations would take us out of the realm of formal power series in $z$. However, in a space of functions, a shift $ z \mapsto z+c $ does preserve the KP hierarchy, so if the formal power series converges to a function on a large enough domain, this shift does preserve KP. This argument is essentially taken from \cite[Section 8]{KazarianHodge}. In particular, under the `natural analytic assumptions' of \cite[section 1.3]{BDKS20a}, i.e. the assumptions in the second part of \cref{OScorrelators}, the $ H_{g,n}$ do extend to rational functions on all of $ \P^1$, so this shift is well-defined.
\end{remark}

\subsection{KP for intersection numbers}\label{KPwithTR}

Now we will restrict to the cases where topological recursion, and hence \cref{EynardDOSS}, can be used, in order to relate to intersection numbers.
In this case, the following holds from \cref{CorrelatorsAsIntersection}.
\begin{equation}
       F (\underline{p})
       = 
      \bigg(  \hbar^{-1} H_{0,1} + \frac{1}{2} H_{0,2} + \!\!\! \sum_{2g-2+n>0} \!\! \frac{\hbar^{2g-2+n}}{n!}\!\!\!\!  \sum_{j_1, \dotsc, j_n \in J} \int_{\Mgn} \!\!\! \Omega_{g,n} (e_{j_1} \otimes \dotsb \otimes e_{j_n}) \prod_{i=1}^n \sum_{k_i=0}^\infty \psi_i^{k_i} \xi_{k_i}^{j_i} (z_i) \bigg) \bigg|_{X_i^{k_i} \to p_{k_i}},
\end{equation}
if we define $ \xi_k^j(z) \coloneqq \int_{z'=r_j}^z d\xi^j_k(z')$, noting that due to the shape of the $H_{g,n}$ in \cref{HnDef,HnGenusDecomp} and $X(z)$ in \cref{XfromKP}, the $H_{g,n}$ have no constant terms in $z_i$.

Under the correspondence $ p_k \leftrightarrow X^k $, $ q_m \leftrightarrow z^m$ of \cref{LinearCorrespondence}, we define $T_k^j$ by
\begin{equation}\label{TsUnderCorrespondence}
	T^j_k(\underline{p}) 
	\leftrightarrow \frac{1}{dx} d\xi_k^j (z) = D^{k+1}  \frac{z^{j+1}}{j+1}\,,
\end{equation}
with $ D$ as in \cref{XfromKP}. Explicitly,
\begin{equation}
	T_{-1}^j = \frac{1}{j+1} q_{j+1}\,,
	\qquad
	T_{k+1}^j = \sum_{m=1}^\infty \sum_{l=0}^\infty m \mc{T}_l q_{m+l} \frac{\del}{\del q_m} T_k^j\,,
	\quad 
	\text{with $\mc{T}_l$ given by}
	\quad
	Q(z)^{-1} = \sum_{l=0}^\infty \mc{T}_l z^l\,.
\end{equation}
Note that, even though the recursion operator for the $ T_k^j$ may have infinitely many terms, its alternate description via \cref{TsUnderCorrespondence} ensures they are well-defined in $ \Lambda$.

Therefore, by definition,
\begin{equation}\label{CohFTGeneratingInp}
	\begin{aligned}
		&
		\Theta \big(\hbar^{-1} H_{0,1} +\frac{1}{2} H_{0,2}\big)  + G_\Omega (T(\underline{p}))
		\\
		& \quad = 
		\Theta \big(\hbar^{-1} H_{0,1} +\frac{1}{2} H_{0,2}\big)  + \!\! \sum_{2g-2+n>0}\!\! \frac{\hbar^{2g-2+n}}{n!} \!\!\!\! \sum_{j_1, \dotsc, j_n \in J} \int_{\Mgn} \!\!\! \Omega_{g,n} (e_{j_1} \otimes \dotsb \otimes e_{j_n}) \prod_{i=1}^n \sum_{k_i=0}^\infty \psi_i^{k_i} T_{k_i}^{j_i}(\underline{p})
	\end{aligned}
\end{equation}
is the logarithm of a tau-function, where $ \{ e_j  \}$ is the dual basis to the basis $ \{ d\xi_0^j\} $  of $ V^*$.

\begin{proposition}\label{ProofKPForBDKSIntNumbers}
	Suppose the pair of functions $ (\hat{\psi}, \hat{y})$ lies in family I or II of \cref{OScorrelators}, and let $\Omega$ be the cohomological field theory associated to the related topological recursion via \cref{EynardDOSS}. Then
	\begin{equation}
	Z_\Omega (\underline{q}) =\exp ( G_{\Omega}(T(p(\underline{q}))))
	\end{equation}
	is a KP tau-function.
\end{proposition}
\begin{proof}
	Apply \cref{KPtauWithoutUnstable} to the exponent of \cref{CohFTGeneratingInp}.
\end{proof}

\subsection{Finiteness of the transformation}\label{Finiteness}

The operator $A$ in the proof of \cref{KPpreservingBDKS} corresponds to a finite sum of $ L_m$ if and only if $ Q(z)^{-1}$ is a polynomial in $z$. As this case seems particularly nice, we will investigate it here. Note that this condition dependent on the parameter $z$ on the spectral curve, cf. the difference between \cref{NaiveHodge} and \cref{RecallHodge}.\par
Write $ P(z) = Q(z)^{-1}$ for this polynomial, and write $r+1$ for its degree. From \cref{XfromKP}, it follows that $ P(0)=1$, so we may write
\begin{equation}
	P(z) = \prod_{j=1}^{r+1} (1- c_j z)\,.
\end{equation}
We immediately see that $ dx(z) = \frac{dz}{z P(z)}$, and hence the spectral curve has a unique ramification point, $\infty$, of ramification index $r$. This is also the rank of the associated Frobenius algebra. But we can do better. 
By calculating the residues in $v$ of
\begin{equation}
\frac{v^{r+1} dv}{(1-vz) \prod_{k=1}^{r+1} (v-c_k)}\,
\end{equation}
and using that they sum to zero, one can check that (if all $ c_j $ are distinct)\footnote{If some $c_j$ coincide, the residue argument still holds, but the result changes.}
\begin{equation}
	\frac{dx}{dz} = \frac{1}{z} + \sum_{j=1}^{r+1} \frac{c_j^{r+1}}{\prod_{k \neq j} (c_j - c_k)} \frac{1}{1-c_jz}\,,
\end{equation}
from which we see that
\begin{equation}
	x(z) = \log z - \sum_{j=1}^{r+1} \prod_{k \neq j} \big( 1- \frac{c_k}{c_j} \big)^{-1} \log (1-c_jz)\,.
\end{equation}
If $r=0$, $dx$ has two (simple) poles, and hence no zeroes. In fact, in this case, $X$ is a Möbius transformation.\par
If $r =1$, this recovers the triple Hodge curve, studied in \cref{THsection} below.\par
If $ r >1$, the Frobenius algebra is not semi-simple: it seems to be a deformation of the algebra corresponding to Witten's $ r+1$-spin cohomological field theory, which is given by $ x = y^{r+1}$, cf. \cite{Wit93,DNOPS16,BCEG21}. This class fits in Alexandrov's theory of the deformed generalised Kontsevich model \cite{Ale21}: it seems like it is a complementary subspace of the polynomial deformations of the Witten $r+1$-spin theory.
\par
Interestingly, except for special choices of $ c_j$, these cases seem not to be covered in the two families in \cref{OScorrelators} for which topological recursion is proved (for any choice of $y$). Even the $r=1$ case does not fall in that scope, unless $ \frac{c_1}{c_2} \in \Q$.

\subsection{The case of triple Hodge integrals}
\label{THsection}

Let us now consider the special case of triple Hodge integrals. The approach taken in this section overlaps with the previous results, but is also slightly different in details, adapted to this specific problem. For example, we do not use the formal parameter $ \hbar $, but make another convenient choice. In this case, the ELSV-type formula is completely explicit, and there is no need to take the detour via topological recursion.\par
The coordinate change we want to perform is inspired by the Mari\~{n}o-Vafa formula.

\begin{equation}
\begin{split}
	F \big(w, \beta; \underline{p} \big) &= \log \bigg( \sum_{m=0}^\infty \sum_{\mu, \nu \vdash m} \frac{\chi^\nu_\mu}{z_\mu} e^{(1+\frac{w}{2})\beta f_2(\nu )} \prod_{\square \in \nu } \frac{\beta w}{\varsigma (\beta w h_\square )} p_\mu \bigg)
	\\
	&
	= \sum_\mu \sum_{g = 0}^\infty \frac{(w + 1)^{g +n-1}}{|\Aut \mu |}   \prod_{i=1}^n \frac{\prod_{j=1}^{\mu_i-1} (\mu_i + j w )}{(\mu_i-1) !} \int_{\Mgn}\!\! \frac{\Lambda (-1) \Lambda (-w ) \Lambda \big(\frac{w}{w + 1}\big)}{\prod_{i=1}^n (1- \mu_i \psi_i )} \beta^{2g-2+n + | \mu |}  p_\mu \,.
\end{split}
\end{equation}

As $ 2g-2+n+|\mu | = \frac{2}{3} \dim \Mgn + \sum_{i=1}^n (\mu_i + \frac{1}{3} )$ and $ g+n-1 = \frac{1}{3} \dim \Mgn + \sum_{i=1}^n \frac{2}{3} $, we get after rewriting $ u \coloneqq \beta^\frac{1}{3}(w +1)^\frac{1}{6} $
\begin{equation}
	\begin{split}
	F \big(w, \beta; \underline{p} \big) &= \sum_\mu  \frac{1}{|\Aut \mu |}\sum_{g = 0}^\infty \prod_{i=1}^n u^4 \frac{\prod_{j=1}^{\mu_i-1} (\mu_i + j w )\beta}{(\mu_i-1) !} \int_{\Mgn} \frac{\Lambda (-u^2) \Lambda (-u^2w) \Lambda \big(\frac{u^2w}{w + 1}\big)}{\prod_{i=1}^n (1- \mu_i u^2\psi_i )} p_\mu 
	\\
	\label{THHurwitztoTHgen} 
	&
	= \sum_{g=0}^\infty \sum_{n=1}^\infty \frac{1}{n!} \int_{\Mgn} \Lambda (-u^2) \Lambda (-u^2w) \Lambda \Big(\frac{u^2w}{w + 1}\Big) \prod_{i=1}^n \sum_{d=0}^\infty T_d (\underline{p}) \psi_i^d 
	\\
	&
	= G_\textup{TH} \Big( -u^2,-u^2w, \frac{u^2w}{w +1} ; T(\underline{p}) \Big) + H_{0,1} + \frac{1}{2} H_{0,2} \,,
	\end{split}
\end{equation}
where
\begin{equation}\label{Tasp}
	T_d (\underline{p})\coloneqq \sum_{m =1}^\infty \frac{\prod_{j=1}^{m-1} (m+ j w )}{(m-1) !} m^d u^{2d+4} \beta^{m-1}p_m\,.
\end{equation}
Hence, our goal is to show that this change of variables and addition of the unstable terms preserves solutions of the KP hierarchy. 

\begin{lemma}\label{Inverseseries}
	The following two expressions are inverse to each other:
	\begin{equation}
		X(z) 
		= \frac{z}{1 + (w + 1)\beta z} \bigg( \frac{1+\beta z}{1+ (w + 1)\beta z}\bigg)^{\frac{1}{w}} \,; 
		\qquad
		z(X) 
		= \sum_{m = 1}^\infty \frac{\prod_{j=1}^{m-1} (m+j w )}{(m-1)!} \beta^{m-1} X^m\,. \label{x-zInversion}
	\end{equation}
	
\end{lemma}

\begin{proof}
	This can be proved by a residue calculation. Start from the formula for $ X(z) $ with $ \beta =1 $ and write $ z(X) = \sum_{m=1}^\infty C_m X^m $. Then $ C_m = \Res_{X=0} z \, X^{-m} \frac{dX}{X} $, and
	\begin{align*}
		\frac{dX}{X} 
		&
		= \frac{dz}{z} + \frac{d(1+z)^\frac{1}{w}}{(1+z)^{\frac{1}{w}}} + \frac{d (1+ (w + 1)z)^{-\frac{w + 1}{w}})}{(1+ (w +1)z)^{-\frac{w +1}{w}}} 
		\\
		&
		= \frac{dz}{z} + \frac{1}{w} \frac{dz}{1+z} - \frac{(w + 1)^2}{w} \frac{ dz}{1+ (w +1)z} 
		\\
		&
		= \frac{dz}{z(1+z)(1+ (w +1)z)}\,.
	\end{align*}
	Therefore,
	\begin{align*}
		C_m 
		&
		= \Res_{X=0}z \, X^{-m} \frac{dz}{z(1+z)(1+ (w +1)z)}
		\\
		&
		= \Res_{z=0} z^{-m} (1+z)^{-\frac{m}{w} -1} (1+ (w +1)z)^{m \frac{w + 1}{w} -1} dz
		\\
		&
		= \Res_{z=0} z^{-m} \sum_{k=0}^\infty \frac{\prod_{i=0}^{k-1} (-\frac{m}{w}-1-i)}{k!} z^k \sum_{l=0}^\infty \frac{\prod_{j=0}^{l-1} (m\frac{w+1}{w}-1-j)}{l!} (w +1)^l z^l dz
		\\
		&
		= \Res_{z=0} z^{-m} \sum_{k=0}^\infty \frac{\prod_{i=1}^{k} (\frac{m}{w}+i)}{k!} (-1)^k z^k \sum_{l=0}^\infty \frac{\prod_{j=m-l}^{m-1} (\frac{m}{w} +j)}{l!} (w +1)^l z^l dz
		\\
		&
		= \sum_{k=0}^{m-1}  \frac{\prod_{i = 1}^k(\frac{m}{w} +i)}{k!} \frac{\prod_{j=k+1}^{m-1}( \frac{m}{w} +j)}{(m-k-1)!}(-1)^k (w +1)^{m-k-1}
		\\
		&
		= \frac{\prod_{j=1}^{m-1} (\frac{m}{w} +j)}{(m-1)!} \sum_{k=0}^{m-1} \binom{m-1}{k} (-1)^k (w +1)^{m-k-1} 
		\\
		&
		= \frac{\prod_{j=1}^{m-1} (\frac{m}{w} +j)}{(m-1)!} w^{m-1} = \frac{\prod_{j=1}^{m-1} (m+jw )}{(m-1)!}\,.
	\end{align*}
	Finally, $ \beta $ can be introduced in this formula by scaling $ z \to \beta z$, $ X \to \beta X $.
\end{proof}
%

\begin{corollary}
	The expressions for $X(z) $ in \cref{Inverseseries} and \cref{MarinoVafaIsBDKS} are related by a Möbius transformation
	\begin{equation}
		z \mapsto \frac{z}{1+(w+1)\beta z} \,.
	\end{equation}
	Hence, by \cref{XMoebius}, they require the same correction term for their induced linear change of variables.
\end{corollary}

We see that in this particular case we may obtain the function $X$ in two different ways: from the general theory of \cref{OScorrelators}, or from the specific shape of the Mariño-Vafa formula, \cref{MVthm}. In fact, the second choice is nothing but choosing the spectral curve coordinate $ z$ to equal $ \xi$ (which is unique in this case), or in other words $ T_0 = q_1$.

%

\begin{remark}
	Under the correspondence of \cref{EynardDOSS}, the rank of the cohomological field theory corresponds to the number of zeroes of $dx$, counted with multiplicities. So for rank one, $ dx$ can only have one zero, and hence must have three poles. By Möbius transformation, we may place the zero at infinity, and two of the poles at $0$ and $ -1$, from which we find that $dx $ must correspond to the $\frac{dX}{X} $ found in the proof of \cref{Inverseseries}. This may explain in part why Alexandrov~\cite{Ale20} finds only the triple Hodge CohFT in the intersection of the orbits of the Givental and Heisenberg-Virasoro groups. However, $dx$ is not the only datum of a spectral curve, and while $ \P^1$ is rigid and has a unique $B$, it is not clear why there is no freedom in the choice of $dy$.
\end{remark}

\begin{lemma}\label{xdiffequation}
	The series $ X(z) $ from \cref{Inverseseries} satisfies the differential equation
	\begin{equation}
		\frac{\del X}{\del \beta} (z) = - \big( (w + 2)z + (w + 1)\beta z^2)z \frac{\del X}{\del z}(z)\,.
	\end{equation}
\end{lemma}

\begin{proof}
	For $ X(z) = \frac{z}{1 + (w + 1) z} \Big( \frac{1+ z}{1+ (w + 1) z}\Big)^{\frac{1}{w}} $, we get $ Q(z)^{-1} = (1+z)(1+(w+1)z)$, which using \cref{FlowXalongbeta} immediately yields the result.
\end{proof}

We use this lemma in combination with the linear correspondence of \cref{LinearCorrespondence}, slightly adapted as follows: define a linear correspondence $ \Theta $ between power series in $ X$ or $ z$ on the one hand and linear series in $ p$ or $ \tilde{q}$ on the other by
\begin{align}
	p_k 
	&\leftrightarrow X^k \,, 
	& \tilde{q}_m 
	&\leftrightarrow z^m\,.
\end{align}

\begin{definition}
	We define a linear morphism between power series in $ \{ p_m \}_{m \geq 1} $ and $ \{ \tilde{q}_d \}_{d \geq 1} $ by
	\begin{align}
		p_k (\tilde{\underline{q}}) 
		&= \sum_{m = k}^\infty c_k^m \tilde{q}_m 
		&\text{with $ c_k^m $ given by}
		&& X^k 
		&= \sum_{m=k}^\infty c_k^m z^m\,.
	\end{align}
\end{definition}

Under the correspondence $ p_k \leftrightarrow X^k $, $ \tilde{q}_m \leftrightarrow z^m$, we have
\begin{align}
	T_d(\underline{p}) 
	&\leftrightarrow (u^2D)^d u^4 z \,; 
	& D 
	& \coloneqq X\frac{\del}{\del X} = \big( 1 + \beta z\big) \big( 1+(w + 1)\beta z\big) z \frac{\del }{\del z}\,.
\end{align}

In terms of $ \tilde{q}$-variables, this gives
\begin{align}
	T_d 
	&= u^2\sum_{m=1}^\infty m\big( \tilde{q}_m + (w +2)\beta \tilde{q}_{m+1} + (w +1)\beta^2 \tilde{q}_{m+2}\big) \frac{\del}{\del \tilde{q}_m} T_{d-1}\,; & T_0 &= u^4 \tilde{q}_1\,.
\end{align}

If we write $ q_m \coloneqq u^{4m}\tilde{q}_m $, and using $ \beta = \frac{u^3}{\sqrt{w +1}} $, this becomes
\begin{align}
	T_d 
	&= \sum_{m=1}^\infty m\Big( u^{4m+2} \tilde{q}_m + (w +2)\frac{u^3}{\sqrt{w +1}} u^{4m+2}\tilde{q}_{m+1} + u^6 u^{4m+2}\tilde{q}_{m+2}\Big) \frac{1}{u^{4m}}\frac{\del}{\del \tilde{q}_m} T_{d-1}
	\\
	&= \sum_{m=1}^\infty m\Big( u^2 q_m + \frac{u(w +2)}{\sqrt{w +1}} q_{m+1} +q_{m+2}\Big) \frac{\del}{\del q_m} T_{d-1}\,; 
	\\
	T_0 
	&= q_1\,.
\end{align}
This is exactly the definition given in \cref{thm:KPforTH}.

%

\begin{corollary}\label{quadraticcorrectionforx}
	For $ X(z) = \frac{z}{1 + (w + 1)\beta z} \big( \frac{1+\beta z}{1+ (w + 1)\beta z}\big)^{\frac{1}{w}}$, the quadratic correction of \cref{changeofvarsforKP} is $ Q = - \frac{1}{2} \Theta (H_{0,2}) $.
\end{corollary}


\begin{proof}
	The function $ X(z)$ satisfies the conditions of \cref{OScorrelators}, so we may apply \cref{KPpreservingBDKS}.
\end{proof}

Now we are ready to prove the main result on KP integrability of triple Hodge integrals.

\begin{proof}[Proof of \cref{thm:KPforTH}]
	By \cref{MarinoVafaIsBDKS}, $ F $ is a solution of the KP hierarchy in the variables $ t_k \coloneqq \frac{p_k}{k} $. In the same way as for \cref{ProofKPForBDKSIntNumbers}, now using \cref{quadraticcorrectionforx,THHurwitztoTHgen}, 
	\begin{equation}
		G_\textup{TH} \Big( -u^2,-w u^2, \frac{w u^2}{w +1} ; \{ T_d(p(\tilde{q})) \}\Big) 
	\end{equation}
	is a solution of the KP hierarchy in the variables $ \frac{\tilde{q}_m}{m} $. As the KP hierarchy is quasi-homogeneous, rescaling $ \tilde{q}_m \to q_m $ preserves solutions. This completes the proof.
\end{proof}

\begin{remark}\label{changeofvarsforw=-1}
	The result in this subsection do hold for $ w = -1$ (ignoring powers of $ u $), but in this specific case $ X( z)$ is a Möbius transformation, so it reduces to the setting of \cref{XMoebius}. From another point of view, in this case the change of coordinates \cref{ptoq} is an isomorphism, whereas it gives a half-dimensional subspace in all other cases. Equations for this half-dimensional space, in the linear Hodge case, were found in~\cite{Ale15}, cf. also~\cite{Guo-Wang-LinearHodge} for a reformulation. These can be viewed as a deformation of the reduction from KP to KdV. Similar equations should exist for triple Hodge integrals as well, but clearly none of this works for $ w = -1$.\par
	In light of \cref{Finiteness}, one may expect a deformation of the reduction from KP to $r$-KdV or $r$-Gelfand-Dickey for the families found there.
\end{remark}

{\setlength\emergencystretch{\hsize}\hbadness=10000
\printbibliography}

\end{document}